\documentclass[11pt,reqno]{amsart}

\allowdisplaybreaks
\usepackage[title]{appendix}
\usepackage{times}
\usepackage{amsmath,amsfonts, amstext,amssymb,amsbsy,amsopn,amsthm}
\usepackage[initials,alphabetic]{amsrefs}
\usepackage{mathrsfs}
\usepackage{bm}
\usepackage{dsfont}
\usepackage{esint}
\usepackage{graphicx}   % for figures
\usepackage{hyperref}
\usepackage[all]{xy}
\usepackage{relsize}
\usepackage{mathtools}
\usepackage{array}
\usepackage{verbatimbox}
\usepackage{xcolor}
\usepackage{dsfont}
\usepackage{graphicx}   % for figures
\usepackage{hyperref}
\usepackage{tikz}
\usepackage{enumitem}

\newcommand{\BR}{\mathbb{R}}
\newcommand{\BC}{\mathbb{C}}
\newcommand{\BZ}{\mathbb{Z}}
\newcommand{\SU}{\operatorname{SU}}

\newcommand{\Sp}{\operatorname{Sp}}

\newcommand{\seq}{\subseteq}
\newcommand{\wt}{\widetilde}
\newcommand{\non}{\nonumber}
\newcommand{\ad}{\operatorname{ad}}
\newcommand{\Ad}{\operatorname{Ad}}
\newcommand{\p}{\partial}

\newcommand{\mcl}{\mathcal}
\newcommand{\Span}{\operatorname{span}}

\let\emptyset\varnothing

\newtheorem{theorem}{Theorem}[section]

\newtheorem{proposition}[theorem]{Proposition}
\newtheorem{lemma}[theorem]{Lemma}

\theoremstyle{definition}
\newtheorem{definition}[theorem]{Definition}
\theoremstyle{remark}

\theoremstyle{remark}

\theoremstyle{remark}

\theoremstyle{remark}

\theoremstyle{remark}

\theoremstyle{remark}

\begin{document}

\title{Partial Regularity of Stable Stationary Harmonic Maps into Compact Lie Groups}
\date{\today}
\author{Jacob Krantz}
\address{Department of Mathematics, Princeton University, Princeton, NJ 08540, USA}
\email{jk9945@princeton.edu}

\maketitle
\begin{abstract}
Let $M$ be a compact Riemannian manifold, and let $G$ be a compact Lie group with bi-invariant metric. We show that the singular set of any stable stationary harmonic map $u : M \to G$ has Hausdorff codimension at least four. This is sharp.
\end{abstract}
\tableofcontents

\section{Introduction}
The regularity of harmonic maps between compact manifolds has been studied very intensively. In dimension at most two, H\'elein \cite{Hel91} showed that weakly harmonic maps are smooth. However, once the dimension of the domain manifold exceeds two, weakly harmonic maps can be singular. Therefore, instead of attempting to derive a complete regularity theory, one attempts to prove partial regularity. 

It turns out that in higher dimensions even partial regularity cannot be expected for weakly harmonic maps. For example, as shown by Rivi\`ere \cite{Riv95}, there are weakly harmonic maps that are everywhere discontinuous. In order to achieve partial regularity, we would then assume that a map is more than weakly harmonic. In particular, we might assume the map is minimizing. In that case, Schoen and Uhlenbeck \cite{SU82} were able to develop a satisfying partial regularity theory. They showed that minimizing harmonic maps between compact manifolds are smooth away from a Hausdorff codimension three set. They also developed a dimension reduction procedure to show that minimizing harmonic maps into special targets are smooth away from even higher codimension sets. In \cite{SU84} they applied this to the case of minimizing harmonic maps into spheres. In particular, they showed that minimizing harmonic maps into $S^{3}$ are smooth away from a Hausdorff codimension four set. Note that $S^{3} = \SU(2).$ This is the starting point of our work. After Schoen and Uhlenbeck's work, several papers addressed similar questions for homogeneous targets. See for example Xin \cite{Xin89} and Okayasu \cite{Oka94}.

In certain cases, a similar dimension reduction argument can be applied in the context of stable stationary harmonic maps. Hong and Wang \cite{HW99} showed that in certain cases stable stationary harmonic maps satisfy a compactness property. Lin and Wang \cite{LW06} used this to study the partial regularity of stable stationary maps into spheres. Recently, Li \cite{Li26} was able to prove the optimal partial regularity result for stable stationary maps into spheres. 

Hsu \cite{Hsu05} showed that stable stationary harmonic maps into targets that admit no nonconstant stable harmonic map from $S^{2}$ satisfy an analogous compactness theorem. In particular, Hsu's theorem will apply in our paper. The form of the theorem that we will use says that as long as every nonconstant harmonic map from $S^{2}$ into our target manifold fails the cone stability inequality (\ref{conestabineq}), we know that every stable stationary harmonic map into the target manifold is smooth away from a Hausdorff codimension four set. This result was used in \cite{Kra26} to show that stable stationary harmonic maps into compact simple Lie groups that are not $\Sp(n)$ for $n \geq 8$, $\operatorname{E}_{8}$, $\operatorname{F}_{4}$, or $\operatorname{G}_{2}$ are smooth away from a Hausdorff codimension four set. Here we generalize this result. We now state our main theorem as follows:

\begin{theorem}\label{mainthm}
Let $u : M \to (G,g)$ be a stable stationary harmonic map from a compact manifold into a compact Lie group with bi-invariant metric. Then $u$ is smooth away from a set of Hausdorff codimension at least four.
\end{theorem}

This is sharp. See \cite{Kra26} for examples of codimension four singularities of stable stationary harmonic maps into compact simple Lie groups. 

Before discussing the proof of Theorem \ref{mainthm}, we emphasize that the proof requires an understanding of harmonic maps $S^{2} \to G.$ The study of such maps was greatly influenced by Uhlenbeck \cite{Uhl89}. Soon after Uhlenbeck's paper, Burstall and Rawnsley \cite{BR90} were able to produce a factorization theorem that turns out to be very helpful in proving Theorem \ref{mainthm}.

This paper will proceed by first recalling the basics and defining cone stability. In this first part of the argument we assume we are working with a compact simple Lie group target with bi-invariant metric. We then proceed by contradiction. We fix a nonconstant cone stable harmonic map $\phi : S^{2} \to G$. This tells us that $\phi$ satisfies the cone stability inequality (\ref{conestabineq}). We attempt to find a variation that fails (\ref{conestabineq}).

To do this, we then review Grothendieck's theorem \cite{Gro57} about holomorphic principal $G^{\BC}$-bundles on $S^{2}.$ This theorem is crucial because it is what allows us to holomorphically decompose $\phi^{*}TG^{\BC}$ into root bundles determined by the roots of $\mathfrak{g}^{\BC}$ with respect to a fixed maximal toral subalgebra $\mathfrak{t} \seq \mathfrak{g}$. Looking carefully at these root bundles, we choose the variation determined by the coroot dual to the highest root. This specific choice of variation will allow us to show that certain potentially positive terms of (\ref{conestabineq}) automatically vanish for our choice of variation. What we are left with are negative terms with a topological upper bound depending only on the first Chern class of the line bundle associated to the highest root. From here we immediately get a contradiction. By Hsu's theorem, this completes the proof in the case of compact simple targets. The general case then follows quickly.

Note that our construction is very similar to the factorization theorem of \cite{BR90}. In particular, the use of Grothendieck's theorem and construction of the variational field above are very similar to the proof of their factorization theorem. The main difference between their setting and ours is that we do not need our variation to induce a harmonic map as in their factorization theorem. Indeed, we do not have to restrict ourselves to variations coming from parabolic subalgebras corresponding to simple roots with coefficient one in the highest root. This observation will be the main motivation for our proof below.

\textbf{AI Acknowledgement}: This argument was assisted by AI. In particular, when attempting to remove the groups remaining in the original paper \cite{Kra26} done without AI, I was unable to show that a certain potentially positive contribution to the second variation vanishes. After discussing this with the AI, it was able to show that this is indeed true. It is proved in Section \ref{b-2vanishessection} below. AI was also used to improve the exposition, although I wrote this myself. Finally, the AI also noticed that this argument can be extended to disconnected compact Lie groups.

\section{Harmonic Map Basics}
For this section, let $(M^{m},g')$ and $(N,g)$ be compact manifolds. We will be studying maps $u : M \to N$ that are critical points of the Dirichlet energy. The Dirichlet energy is defined as:
\begin{align}
E[u] = \dfrac{1}{2}\int_{M}|\nabla u|^{2}
\end{align}
This energy makes sense for smooth maps between manifolds, but care is needed to define the energy for weakly differentiable maps. The standard way to handle this problem is to fix an isometric embedding $\iota : N \to \BR^{k}.$ Then one defines
\begin{align}
W^{1,2}(M,N) = \{u \in W^{1,2}(M,\BR^{k}): u(x) \in N \text{~for a.e.~} x \in M\}.
\end{align}
Next, we define the singular set of a map $u \in W^{1,2}(M,N)$ to be the complement of the regular set of $u$. Here, the regular set of $u$ is defined as:
\begin{align}
\operatorname{Reg}(u) = \{x \in M : u \text{~is smooth in some ball around~} x\}.
\end{align}
We now define what it means for a $W^{1,2}(M,N)$ map to be weakly harmonic.
\begin{definition}[Weakly Harmonic Map] The map $u \in W^{1,2}(M,N)$ is said to be weakly harmonic if it is a critical point of the Dirichlet energy with respect to target variations. This can be expressed as the weak equality:
\begin{align}
\Delta u = A_{u}(\nabla u, \nabla u)
\end{align}
Here $A$ is the second fundamental form of our isometric embedding.
\end{definition}
This paper will also assume harmonic maps are stationary because weakly harmonic maps that are not stationary cannot be expected to exhibit partial regularity. Note however that if our map is smooth, weak and stationary harmonicity necessarily coincide.
\begin{definition}[Stationary Harmonic Map]
We say that a weakly harmonic map $u$ is stationary if it is also critical with respect to domain variations. This condition produces the following weak identity:
\begin{align}
\operatorname{div}\left (\dfrac{1}{2}|\nabla u|^{2}g' - u^{*}g\right ) = 0.
\end{align}
\end{definition}
While stationary harmonic maps are smooth away from a set of $(m-2)$-dimensional Hausdorff measure zero \cite{Bet93}, it is impossible to reduce the dimension of the singular set any further in general \cite{JLY26}. Therefore, one might work with stable stationary maps.
\begin{definition}[Stable Harmonic Map]
We say that a weakly harmonic map $u$ is stable if 
\begin{align}
\int_{M}|\nabla^{u}X|^{2} - g(R^{N}(X,\nabla_{i}u)\nabla_{i}u,X) \geq 0
\end{align}
for every compactly supported $X \in \Gamma(M,u^{*}TN).$ Here $\nabla^{u}$ is the pullback of the Levi-Civita connection on $TN$ by $u.$
\end{definition}
We call a map $u \in W^{1,2}(M,N)$ stable stationary if it is stable and stationary. We will study stable stationary maps in this paper.
\section{The Cone Stability Inequality and Dimension Reduction}
Here we state the important analytic preliminaries. Note that we will always assume $S^{2}$ is equipped with the round metric of radius one. When equipped with its usual complex structure, $S^{2}$ admits a local holomorphic coordinate $z$ for which the metric is conformal to the Euclidean metric. We call the corresponding real coordinates $(x,y).$ These coordinates will at times be useful. We now move on to an important definition.

\begin{definition}[Cone Stable Map]
Let $\phi : S^{2} \to (N,g)$ be a harmonic map. We call the homogeneous extension of $\phi$, $\wt\phi : \BR^{3} \to N$, the cone map associated to $\phi.$ We say that $\phi$ is cone stable if $\wt\phi$ is a stable harmonic map. Note that $\phi$ is weakly harmonic if and only if $\wt\phi$ is weakly harmonic.
\end{definition}

We will next discuss dimension reduction for stable stationary harmonic maps. The particular result that we will use is due to Hsu \cite{Hsu05} and can be written as follows:

\begin{theorem}[\cite{Hsu05}]
Let $u : M \to N$ be a stable stationary harmonic map. Suppose there are no nonconstant stable harmonic maps $S^{2} \to N.$ Suppose furthermore that there are no nonconstant cone stable harmonic maps $S^{2} \to N.$ Then $u$ is smooth away from a Hausdorff codimension four set.
\end{theorem}

It is well-known that there are no nonconstant stable harmonic maps $S^{2} \to G$ (see \cite{BR90}, Theorem $7.5$). Therefore, a nonexistence theorem for nonconstant cone stable maps from $S^{2}$ into $G$ automatically yields our main theorem. We know a lot about harmonic maps from $S^{2}$, but cone stability is a condition on an associated map from $\BR^{3}.$ Therefore, we would like to write cone stability as a condition on the underlying map from $S^{2}.$ It turns out that if $\phi$ is cone stable, it must satisfy a ``cone stability inequality:"

\begin{proposition}[Schoen-Uhlenbeck]
Let $\phi : S^{2} \to (N,g)$ be a cone stable harmonic map. Then, 
\begin{align}
\dfrac{1}{4}\int_{S^{2}}|X|^{2} + \int_{S^{2}}|\nabla^{\phi}X|^{2} - g(R^{N}(X,\nabla_{i}\phi)\nabla_{i}\phi,X) \geq 0 \label{conestabineq}
\end{align}
for every $X \in \Gamma(S^{2},\phi^{*}TN).$ Here $\nabla^{\phi}$ is the pullback of the Levi-Civita connection on $N$ by $\phi.$
\end{proposition}

This idea appears to be due to Schoen and Uhlenbeck \cite{SU84}, but the derivation of this general form of the inequality is done in \cite{Kra26}.

Our goal now is to show that for every nonconstant harmonic map $\phi : S^{2} \to G$ we can find a variation $X$ that fails (\ref{conestabineq}). 

\section{Lie Algebra Preliminaries}\label{liealgebrasection}
In this section, we will review some information about Lie algebras. For more information, see Fulton and Harris \cite{FH91} or Burstall and Rawnsley \cite{BR90}.

We will start by assuming $G$ is a compact simple Lie group and will return to the general case at the end of the paper. We will always assume $G$ has the bi-invariant metric given by negative one times the normalization of the Killing form that measures the length squared of long roots as two. We call this $-\kappa.$ Note that this is just to make computations easier and does not affect the result at all. Now, the Lie algebra of $G$ will be denoted $\mathfrak{g}$. The complexification of $G$ will be denoted $G^{\BC}$, and the complexification of $\mathfrak{g}$ will be denoted $\mathfrak{g}^{\BC}.$ $\tau$ will denote complex conjugation (the Cartan involution) in $\mathfrak{g}^{\BC}$ with respect to our fixed compact real form $\mathfrak{g}$, and we will use $\dagger$ to denote $-\tau.$ In particular, $\dagger$ acts as $-\operatorname{Id}$ on $\mathfrak{g}.$ This allows us to define a Hermitian inner product on $\mathfrak{g}^{\BC}$ by $h(X,Y) = \kappa(X,Y^{\dagger}).$ Finally, let $\mathfrak{t} \seq \mathfrak{g}$ be a maximal toral subalgebra. Its complexification $\mathfrak{t}^{\BC} \seq \mathfrak{g}^{\BC}$ is a Cartan subalgebra of $\mathfrak{g}^{\BC}$.

With this notation in hand, we define the roots of $\mathfrak{g}^{\BC}$ with respect to $\mathfrak{t}^{\BC}:$

\begin{definition}[Roots]
Let $\alpha \in (\mathfrak{t}^{\BC})^{*}$ be non-zero. Then set
\begin{align}
\mathfrak{g}^{\alpha} = \{X \in \mathfrak{g}^{\BC} : \text{~for all~} H \in \mathfrak{t}^{\BC}, \ad(H)X = \alpha(H)X\}
\end{align}
If $\mathfrak{g}^{\alpha} \neq \{0\}$, we call $\alpha$ a root of $\mathfrak{g}^{\BC}$ with root space $\mathfrak{g}^{\alpha}.$ The set of roots is denoted $\Delta(\mathfrak{g}^{\BC},\mathfrak{t}^{\BC})$. If the Lie algebra and Cartan subalgebra are already clear, we may denote the set of roots $\Delta.$
\end{definition}

Note that $\Delta \seq \sqrt{-1}\mathfrak{t}^{*}$ and that
\begin{align}
\mathfrak{g}^{\BC} = \mathfrak{t}^{\BC} \oplus \sum_{\alpha \in \Delta}\mathfrak{g}^{\alpha}. \label{gcdecomp}
\end{align}
It is now useful to divide the roots into two subsets:

\begin{definition}[Positive Root System]
Let $\Delta^{+} \seq \Delta$. We say that $\Delta^{+}$ is a choice of positive root system if 
\begin{itemize}
\item $\Delta^{+} \cap -\Delta^{+} = \emptyset$
\item $\Delta^{+} \cup -\Delta^{+} = \Delta$
\item If $\alpha, \beta \in \Delta^{+}$ and $\alpha + \beta \in \Delta$, then $\alpha + \beta \in \Delta^{+}.$
\end{itemize}
\end{definition}

Once we have a choice of positive roots, we get a further subset called the simple roots:

\begin{definition}[Simple Roots]
Any positive root that cannot be expressed as the sum of two other positive roots is called a simple root.
\end{definition}

Every root can be decomposed into an integer linear combination of simple roots. For positive roots, the coefficients are non-negative. 

Next, we define the highest root. Said imprecisely, with respect to a chosen positive root system the highest root is the one with the highest positive coefficients in the simple roots.

In other words, it can be shown that there is a unique positive root that is maximal with respect to the partial ordering on $\Delta^{+}$:
\begin{align}
\alpha \leq \beta \iff \beta - \alpha = 0 \text{~or~} \beta - \alpha \text{~is a positive linear combination of positive roots.}
\end{align}

\begin{definition}[Highest Root]
The highest root of $\Delta^{+}$ is the maximal root with respect to the above partial ordering. We denote it by $\theta.$
\end{definition}

Let $H$ be the dual element to $\theta$ via $\kappa$. In particular, $\theta(X) = \kappa(H,X)$ for all $X \in \mathfrak{t}^{\BC}.$ The normalization $\kappa$ we have chosen implies then that $\theta(H) = \kappa(H,H) = 2.$

We now choose an element $e \in \mathfrak{g}^{\theta}$ such that $H = [e,e^{\dagger}].$ This takes a little bit of work, but it is all standard. 

The first step is showing $[\mathfrak{g}^{\theta},\mathfrak{g}^{-\theta}] = \Span_{\BC}\{H\}.$ To do this, let $X \in \mathfrak{g}^{\theta}$, $Y \in \mathfrak{g}^{-\theta}$, and $Z \in \mathfrak{t}^{\BC}.$ Then,
\begin{align}
[Z,[X,Y]] = -[X,[Y,Z]] - [Y,[Z,X]] = -\theta(Z)[X,Y] + \theta(Z)[X,Y] = 0.
\end{align}
We know $\mathfrak{t}^{\BC}$ is maximally abelian, so $[X,Y] \in \mathfrak{t}^{\BC}.$ Then, we compute:
\begin{align}
\kappa([X,Y], Z) = \kappa(Y,[Z,X]) = \kappa(Y,\theta(Z)X) = \kappa(X,\kappa(H,Z)Y) = \kappa(\kappa(X,Y)H,Z)
\end{align}
so we conclude $[X,Y] = \kappa(X,Y)H$ by the nondegeneracy of $\kappa$ on the Cartan subalgebra $\mathfrak{t}^{\BC}$. This implies $[X,Y]$ is a scalar multiple of $H.$ 

Next, take any non-zero element $e' \in \mathfrak{g}^{\theta}.$ Then, $[e',(e')^{\dagger}] = zH.$ On the other hand,
\begin{align}
\kappa([e',(e')^{\dagger}],H) = \kappa((e')^{\dagger}, 2e') = 2|e'|^{2} > 0.
\end{align}
These two observations imply $z$ is a real number bigger than zero. Therefore we may normalize $e'$ to achieve $e$ satisfying $H = [e,e^{\dagger}].$ This will be very important when constructing our variation later. We might sometimes refer to $e^{\dagger}$ as $f \in \mathfrak{g}^{-\theta}.$ Denote by $V$ the $\mathfrak{sl}(2;\BC)$ subalgebra of $\mathfrak{g}^{\BC}$ generated by $e$, $f$, and $H$.

We will also need to understand parabolic subalgebras and flag manifolds.

\begin{definition}[Parabolic Subalgebra]
Let $\mathfrak{q} \seq \mathfrak{g}^{\BC}$ be a subalgebra. We say $\mathfrak{q}$ is parabolic if it contains a maximal solvable subalgebra.
\end{definition}

\begin{definition}[Parabolic Subgroup]
We say that $P \seq G^{\BC}$ is a parabolic subgroup if its Lie algebra is parabolic. Parabolic subgroups are necessarily connected.
\end{definition}

\begin{definition}[Flag Manifold]
Let $P \seq G^{\BC}$ be a parabolic subgroup. We say a quotient $G^{\BC}/P$ is a flag manifold. It is naturally a complex manifold. Note that $G$ acts transitively on $G^{\BC}/P$, so $G^{\BC}/P$ is diffeomorphic to the real quotient $G/(G \cap P).$
\end{definition}

We will now compute some useful facts about the adjoint action of $\mathfrak{g}^{\BC}$ on the highest root space. In particular, we show the stabilizer is parabolic.

It will be helpful to break our Lie algebra up into components depending on the eigenvalues of $\ad(H).$ Namely, define
\begin{align}
\mathfrak{g}_{k} = \{X \in \mathfrak{g}^{\BC} : \ad(H)X = kX\}
\end{align}
Clearly, $\mathfrak{t}^{\BC} \seq \mathfrak{g}_{0}.$ Also, each root space lies entirely inside a single $\mathfrak{g}_{k}.$ 

We now claim that the only non-zero $\mathfrak{g}_{k}$ have $|k| \leq 2.$ To see this, note that on $\mathfrak{g}^{\alpha}$, $k = \alpha(H)$. This tells us that
\begin{align}
|k| = |\alpha(H)| = |\kappa(H,H_{\alpha})| = |\langle \theta, \alpha \rangle| \leq (\sqrt{2})^{2} = 2
\end{align}
If we have equality above, then $\alpha = \pm \theta$. One can also show that for all $\alpha \in \Delta$, $\alpha(H) \in \BZ$ (see for instance Lecture 11 of \cite{FH91}). It then follows from (\ref{gcdecomp}) that:
\begin{align}
\mathfrak{g}^{\BC} = \Span_{\BC}\{f\} \oplus \mathfrak{g}_{-1} \oplus \mathfrak{g}_{0} \oplus \mathfrak{g}_{1} \oplus \Span_{\BC}\{e\}
\end{align}

With this said, we now want to compute the stabilizer of the action of $\ad(\mathfrak{g}^{\BC})$ on $\Span_{\BC}\{e\}.$

\begin{lemma}\label{stablemma}
One has $\operatorname{Stab}_{\mathfrak{g}^{\BC}}(\Span_{\BC}\{e\}) = \mathfrak{g}_{0} \oplus \mathfrak{g}_{1} \oplus \mathfrak{g}_{2}$, and $\mathfrak{g}_{0} = \Span_{\BC}\{H\} \oplus \mathfrak{c}$, where $\mathfrak{c} = Z_{\mathfrak{g}^{\BC}}(V).$
\end{lemma}
\begin{proof}
First, we show $\mathfrak{g}_{0} = \Span_{\BC}\{H\} \oplus \mathfrak{c}$. Suppose $X \in \operatorname{span}_{\BC}\{H\} \oplus \mathfrak{c}.$ Then, $[H,X] = 0$. Therefore, $\operatorname{span}_{\BC}\{H\} \oplus \mathfrak{c} \seq \mathfrak{g}_{0}.$ On the other hand, say $X$ is an element of $\mathfrak{g}_{0}$ such that $\kappa(X, H) = 0$ (note $H = H^{\dagger}$, so this is orthogonality with respect to the Hermitian inner product). We show $X \in Z_{\mathfrak{g}^{\BC}}(V)$. To this end, we compute
\begin{align}
[X,H] &= 0 \\ \non \\
[H,[X,e]] &= [[e,H],X] + [[H,X],e] = 2[X,e] \\ \non \\
[H,[X,f]] &= -[X,[f,H]] - [f, [H,X]] = -2[X,f]
\end{align}
Then, $[X,e] \in \mathfrak{g}_{2}.$ Therefore, $[X,e] = ce.$ Similarly, $[X,f] \in \mathfrak{g}_{-2}.$ Then $[X,f] = c'f.$ Now, 
\begin{align}
\kappa([X,e],[X,f]) &= \kappa(e, -(c')^{2}f) = -(c')^{2}|e|^{2} \\ \non \\
\kappa([X,e],[X,f]) &= \kappa(-c^{2}e,f) = -c^{2}|e|^{2} \\ \non \\
\kappa([X,e],[X,f]) &= \kappa(X,c'[e,f]) = c'\kappa(X,H) = 0.
\end{align}
Therefore, $c = c' = 0$; we verify that $X \in Z_{\mathfrak{g}^{\BC}}(V).$ Therefore, $\mathfrak{g}_{0} = \operatorname{span}_{\BC}\{H\} \oplus \mathfrak{c}.$ 

To find the entire stabilizer of $\Span_{\BC}\{e\}$, note that if $X \in \mathfrak{g}_{k}, Y \in \mathfrak{g}_{\ell}$, then 
\begin{align}
[H,[X,Y]] = -[X,[Y,H]] - [Y,[H,X]] = (k + \ell)[X,Y] \label{k+ellbracketid}
\end{align}
i.e., $[X,Y] \in \mathfrak{g}_{k + \ell}.$ Therefore, $\mathfrak{g}_{0}, \mathfrak{g}_{1}$, and $\mathfrak{g}_{2}$ are in the stabilizer. 

We then show that none of the other $\mathfrak{g}_{j}$ can contribute to the stabilizer. By (\ref{k+ellbracketid}), it suffices to show that $\ad(e)|_{\mathfrak{g}_{-1}}$ and $\ad(e)|_{\mathfrak{g}_{-2}}$ have trivial kernel. Note that by linearity it suffices to check the adjoint action of $e$ itself.

We focus first on $\mathfrak{g}_{-1}.$ Let $X_{-1} \in \mathfrak{g}_{-1}.$ Then, if $[X_{-1},e] = 0$, we must have
\begin{align}
0 = [[X_{-1},e],f] = -[[f,X_{-1}],e] - [[e,f],X_{-1}] = -[[f,X_{-1}],e] + X_{-1}.
\end{align}
In other words, by (\ref{k+ellbracketid}),
\begin{align}
X_{-1} = [[f,X_{-1}],e] = [0,e] = 0.
\end{align}
Therefore, $\ker(\ad(e)|_{\mathfrak{g}_{-1}}) = \{0\}$. This means that $\mathfrak{g}_{-1}$ cannot contribute to the stabilizer.

The last case to check is $\mathfrak{g}_{-2}.$ This turns out to be the easiest. Let $X_{-2} \in \mathfrak{g}_{-2}.$ Then $X_{-2} = cf$ for some constant $c.$ Then, $\ad(cf)e = -cH$. This is zero only when $c = 0.$ Therefore, $\ker(\ad(e)|_{\mathfrak{g}_{-2}}) = \{0\}.$ This finishes the proof.
\end{proof}
Set $\mathfrak{q}_{\theta} = \operatorname{Stab}_{\mathfrak{g}^{\BC}}(\Span_{\BC}\{e\})$. We see that $\mathfrak{q}_{\theta}$ is parabolic because it contains the Borel subalgebra 
\begin{align}
\mathfrak{t}^{\BC} \oplus \sum_{\alpha \in \Delta^{+}}\mathfrak{g}^{\alpha} \seq \mathfrak{g}^{\BC}
\end{align}
This means that the corresponding subgroup is also parabolic. In what follows, we call this subgroup $P_{\theta}.$ Note also that for the $\Ad$ action we have $\operatorname{Stab}_{G^{\BC}}(\Span_{\BC}\{e\}) = P_{\theta}.$
\section{Holomorphic Bundles}
To begin, we recall a special case of a theorem of Koszul and Malgrange. See also \cite{BR90} Chapter 2 for more details.
\begin{theorem}[Koszul and Malgrange \cite{KM58}]
A connection on a principal bundle with complex structure group over $S^{2}$ determines a unique compatible holomorphic structure. The same result holds for complex vector bundles over $S^{2}.$
\end{theorem}
Consider the product principal $G^{\BC}$-bundle $\pi : P = S^{2} \times G^{\BC} \to S^{2}.$ We equip this bundle with the holomorphic structure induced via the $\mathfrak{g}$-valued one-form $C = \frac{1}{2}\alpha$, where $\alpha = \phi^{*}\omega$ is the pullback of the left Maurer-Cartan form $\omega$ of $G$ by a nonconstant harmonic map $\phi : S^{2} \to G$. This connection on $P$ induces a connection and hence holomorphic structure on $\ad P.$

Now, because $P$ is the product bundle, we have a smooth global section 
\begin{align}
s : S^{2} \to S^{2} \times G^{\BC} : x \mapsto (x,1_{G^{\BC}}).
\end{align}
This induces a smooth isomorphism of $\ad P$ and $S^{2} \times \mathfrak{g}^{\BC}.$ Explicitly it is given by 
\begin{align}
T : \ad P \to S^{2} \times \mathfrak{g}^{\BC} : [sg,\xi] \mapsto (\pi(s),\Ad(g)\xi)
\end{align}
The connection induced on the product bundle $S^{2} \times \mathfrak{g}^{\BC}$ is $D = d + \ad(C).$ The crucial property of this connection is that it coincides with the complex linear extension of the pullback connection on $\phi^{*}TG^{\BC}.$ Therefore, the holomorphic structures on $\ad P$, $S^{2} \times \mathfrak{g}^{\BC}$, and $\phi^{*}TG^{\BC}$ are all equivalent. In particular, it will eventually suffice to work on the product bundle $S^{2} \times \mathfrak{g}^{\BC}.$ 

We have two main structural equations for $\alpha$. The first is the zero-curvature equation:
\begin{align}
d\alpha + \dfrac{1}{2}[\alpha,\alpha] = 0 \implies \p_{z}\alpha_{\bar{z}} - \p_{\bar{z}}\alpha_{z} + [\alpha_{z},\alpha_{\bar{z}}] = 0.
\end{align}

Also, the harmonic map equation can be written:
\begin{align}
d^{*}\alpha = 0 \implies \p_{z}\alpha_{\bar{z}} + \p_{\bar{z}}\alpha_{z} = 0.
\end{align}

Set now $A = \frac{1}{2}\alpha_{z}$, where $\alpha_{z}$ is the $z$-component of $\alpha$ when extended complex linearly. Note that $D_{\bar{z}} = \p_{\bar{z}} + \frac{1}{2}\ad(\alpha_{\bar{z}}).$ Subtracting the zero-curvature equation and the harmonic map equation, we compute that
\begin{align}
D_{\bar{z}}\alpha_{z} = \p_{\bar{z}}\alpha_{z} + \dfrac{1}{2}\ad(\alpha_{\bar{z}})\alpha_{z} = 0.
\end{align}
In other words, $D_{\bar{z}}A = 0.$

Next, let us compute the curvature of $C.$ Note that $C$ is the connection on $P$ written in the gauge $s$. Recalling the definition of curvature, we see that locally:
\begin{align}
F_{C} &= F_{z\bar{z}}dz\wedge d\bar{z} \\ \non \\
&= \left (\dfrac{1}{2}\p_{z}\alpha_{\bar{z}} - \dfrac{1}{2}\p_{\bar{z}}\alpha_{z} + \dfrac{1}{4}[\alpha_{z},\alpha_{\bar{z}}]\right )dz\wedge d\bar{z} \\ \non \\
&= \left (-\dfrac{1}{2}[\alpha_{z},\alpha_{\bar{z}}] + \dfrac{1}{4}[\alpha_{z},\alpha_{\bar{z}}]\right )dz \wedge d\bar{z} \\ \non \\
&= \left (-\dfrac{1}{4}[\alpha_{z},\alpha_{\bar{z}}]\right )dz\wedge d\bar{z} \\ \non \\
&= [A,A^{\dag}]dz\wedge d\bar{z}
\end{align}
Here we used that $\alpha_{z}^{\dagger} = -\alpha_{\bar{z}}$ because $\alpha$ is the complex linear extension of a real form. Now set $\Phi = Adz$, and let $K$ be the canonical bundle of $S^{2}$. The above says that $\Phi \in H^{0}(S^{2},\ad P \otimes K)$. Note that $K = \mcl O(-2)$ in our case.

\section{Grothendieck's Theorem}

We begin this section with a statement of Grothendieck's Theorem about holomorphic principal bundles over $S^{2}$. We refer the reader to Grothendieck \cite{Gro57} or Burstall and Rawnsley \cite{BR90} chapter six for more details. In particular, Burstall and Rawnsley introduce the following root bundle constructions in their work.

\begin{theorem}[Grothendieck's Theorem, \cite{Gro57}]
Let $P \to S^{2}$ be a holomorphic principal $G^{\BC}$-bundle, where $G^{\BC}$ is a complex reductive Lie group. Let $Y \to S^{2}$ be the principal $\BC^{*}$-bundle defined by removing the zero section of $\mcl O(1)$. There exists a holomorphic homomorphism $\chi : \BC^{*} \to G^{\BC}$ such that $P$ is holomorphically isomorphic to $Y \times_{\chi} G^{\BC}$. Moreover, $\chi$ is unique up to conjugation.
\end{theorem}

Let now $P \to S^{2}$ be as in our last section. It is then a standard fact that $\chi$ may be written in the form 
\begin{align}
\chi : \BC^{*} \to G^{\BC} : \exp(\sqrt{-1}z) \mapsto \exp_{G^{\BC}}(z\xi)
\end{align}
for some element $\xi \in \mathfrak{g}^{\BC}.$ It turns out that we can conjugate $\chi$ so that $\xi$ is an element of our favorite maximal toral subalgebra $\mathfrak{t} \seq \mathfrak{g}.$ By construction, $\chi$ then factors through $T^{\BC} = \exp_{G^{\BC}}(\mathfrak{t}^{\BC}).$ Fix a choice of positive roots $\Delta^{+}$ with respect to $\mathfrak{t}^{\BC}$. Up to the action of the Weyl group, $\xi$ satisfies $\alpha(\xi)/\sqrt{-1} \geq 0$ for all positive roots $\alpha$. Next, we will do a root line bundle decomposition of $Y \times_{\Ad \circ \chi} \mathfrak{g}^{\BC}.$

Define the holomorphically trivial sub-bundle
\begin{align}
\mathfrak{a} = Y \times_{\Ad \circ \chi} \mathfrak{t}^{\BC} \seq Y \times_{\Ad \circ \chi} \mathfrak{g}^{\BC}
\end{align}
This is holomorphically trivial because $\Ad|_{T^{\BC}}$ acts as the identity on $\mathfrak{t}^{\BC}$. We then get holomorphic root sub-bundles of $Y \times_{\Ad \circ \chi} \mathfrak{g}^{\BC}$: 
\begin{align}
L^{\alpha} = Y \times_{\Ad \circ \chi} \mathfrak{g}^{\alpha}
\end{align}
where $\mathfrak{g}^{\alpha}$ is the root space corresponding to $\alpha \in \Delta(\mathfrak{g}^{\BC},\mathfrak{t}^{\BC})$.

Finally, one can easily compute the first Chern class of each of the root line bundles. Note that the associating map is $\exp(\sqrt{-1}z) \mapsto \exp_{G^{\BC}}(z\xi) \mapsto \Ad(\exp_{G^{\BC}}(z\xi))$, which evaluates to $\alpha(\exp_{G^{\BC}}(z\xi))X$ on $X \in \mathfrak{g}^{\alpha}.$ Hence, it is $e^{\sqrt{-1}z} \mapsto e^{\sqrt{-1}z\alpha(\xi)/\sqrt{-1}}.$ In other words it is the map $\BC^{*} \to \BC^{*} : w \mapsto w^{\alpha(\xi)/\sqrt{-1}}.$ 

Here, $\alpha(\xi)/\sqrt{-1} \in \BZ$. This is because $\exp_{G^{\BC}}(2\pi\xi) = 1_{G^{\BC}}$ and hence 
\begin{align}
X_{\alpha} = \Ad(\exp_{G^{\BC}}(2\pi \xi))X_{\alpha} = e^{2\pi \alpha(\xi)}X_{\alpha}. 
\end{align}
Notice also that we have identified $\alpha$ with its associated character $T^{\BC} \to \BC^{*}$ when deducing the associating map. 

In what follows, we will use the convention that $c_{1}(\mcl O(1))[S^{2}] = 1.$ We will also denote $c_{1}(E)[S^{2}]$ by just $c_{1}(E)$ for line bundles $E \to S^{2}.$ 

From the above constructions, one sees that $Y \times_{\Ad \circ \chi} \mathfrak{g}^{\BC}$ is holomorphically isomorphic to a direct sum of the form:

\begin{align}
\mathfrak{a} \oplus \sum_{\alpha \in \Delta}L^{\alpha} = \mcl O^{\oplus r} \oplus \sum_{\alpha \in \Delta}\mcl O(\alpha(\xi)/\sqrt{-1})
\end{align}
where $r$ is the rank of $G$, i.e., the dimension of a maximal torus. Now note that $Y \times_{\Ad \circ \chi} \mathfrak{g}^{\BC}$ is isomorphic to $\ad P$ by the above theorem. Therefore, the above decomposition of $Y \times_{\Ad \circ \chi} \mathfrak{g}^{\BC}$ induces a holomorphic direct sum decomposition of $\ad P.$ From the decomposition above we can conclude that because $\Phi$ is a holomorphic section of $\ad P \otimes \mcl O(-2)$ it must lie entirely in the summands where $\alpha(\xi)/\sqrt{-1} \geq 2.$ This is because
\begin{align}
H^{0}(S^{2},\ad P \otimes \mcl O(-2)) &= H^{0}\left (S^{2}, \mcl O(-2)^{\oplus r} \oplus \sum_{\alpha \in \Delta}\mcl O(\alpha(\xi)/\sqrt{-1} - 2)\right ) \\
&= H^{0}(S^{2},\mcl O(-2))^{\oplus r} \oplus \sum_{\alpha \in \Delta}H^{0}(S^{2}, \mcl O(\alpha(\xi)/\sqrt{-1} - 2))
\end{align}
and the only holomorphic line bundles with nonzero holomorphic sections over $S^{2}$ are $\mcl O(k)$ for $k \geq 0.$

Now we arrive at the key point. Let $\theta$ be the highest root, then one has that $\theta - \alpha$ is a non-negative integer linear combination of positive roots for any $\alpha \in \Delta^{+}$. Then, $(\theta - \alpha)(\xi)/\sqrt{-1} \geq 0$. In particular, suppose $\alpha$ is one of the roots corresponding to a nonzero component of $\Phi$. Such a component exists because $\phi$ is nonconstant. Then we conclude $\theta(\xi)/\sqrt{-1} \geq \alpha(\xi)/\sqrt{-1} \geq 2.$ Thus, 
\begin{align}
c_{1}(L^{\theta}) \geq 2. \label{c_1(L)>=2}
\end{align}
From here on, we denote $L = L^{\theta}$ to emphasize that this is the most important root sub-bundle in our construction. Also, because $Y \times_{\Ad \circ \chi} \mathfrak{g}^{\BC}$ is isomorphic to $\ad P$, we will write $L$ as the corresponding holomorphic line bundle in $\ad P.$

Next, notice that the Lie bracket induces a bracket on $\ad P.$ Namely, let $[p,\xi]$ and $[p,\eta]$ be elements of $(\ad P)_{x}.$ One has $[[p,\xi], [p,\eta]] = [p,[\xi,\eta]]$. This is clearly independent of $p$ because $\Ad$ commutes with the Lie bracket. One can check this gives a global bracket on $\ad P$. It also agrees with the analogous bracket on $Y \times_{\Ad \circ \chi} \mathfrak{g}^{\BC}$. We will make precise all of the isomorphisms that we used above in (\ref{grothendieckisomorphisms}) of the following section. Indeed, one can check that brackets are preserved.

A special consequence of the above construction is that for roots $\alpha \neq -\beta$ in $\Delta$, the relation $[\mathfrak{g}^{\alpha},\mathfrak{g}^{\beta}] \seq \mathfrak{g}^{\alpha + \beta}$ is preserved at the bundle level. We may then conclude that 
\begin{align}
[\Phi,L] = 0 \label{phibracketl}
\end{align}
because only positive roots contribute to $\Phi$ and the sum of a positive root with the highest root has trivial root space.

\section{Constructing the Variation}
We will now construct our variation. To begin, we will construct local gauges in which we can understand the root line bundle $L \seq \ad P$ in terms of our compact group $G$. We want to use the compact group specifically so our variation has computable norm.

To this end we first define a map $q : S^{2} \to \mathbb{P}(\mathfrak{g}^{\BC})$ by associating a complex line $L_{x} \seq \mathfrak{g}^{\BC}$ to each $x \in S^{2}.$ To make this identification, we transfer all of our constructions above to the product bundle $S^{2} \times \mathfrak{g}^{\BC} \to S^{2}.$

To do this, we compose all of our various bundle isomorphisms from the previous section. We recall that our root line bundles $L^{\alpha}$ are defined as bundles associated to $Y.$ Therefore, we begin with a bundle $L^{\alpha} = Y \times_{\Ad \circ \chi}\mathfrak{g}^{\alpha}.$ This bundle naturally sits inside $Y \times_{\Ad \circ \chi}\mathfrak{g}^{\BC}$ as a holomorphic sub-bundle. Now, $Y \times_{\Ad \circ \chi} \mathfrak{g}^{\BC}$ is isomorphic to $(Y \times_{\chi} G^{\BC}) \times_{\Ad} \mathfrak{g}^{\BC}$, and $Y \times_{\chi} G^{\BC}$ is isomorphic to $P$ via Grothendieck's theorem. This means $\ad P$ can be identified with $Y \times_{\Ad \circ \chi} \mathfrak{g}^{\BC}$. Hence, $L^{\alpha}$ is identified with a holomorphic sub-bundle of $\ad P.$ Now, identifying $\ad P$ with the product bundle we get holomorphic root line sub-bundles of the product bundle. Note that the holomorphically trivial sub-bundle of Cartan subalgebras can similarly be viewed as a holomorphic sub-bundle of the product bundle. Explicitly, the identification is:
\begin{align}
L^{\alpha} = Y \times_{\Ad \circ \chi} \mathfrak{g}^{\alpha} &\hookrightarrow Y \times_{\Ad \circ \chi} \mathfrak{g}^{\BC} = (Y \times_{\chi} G^{\BC}) \times_{\Ad} \mathfrak{g}^{\BC} = \ad P = S^{2} \times \mathfrak{g}^{\BC} \\
[y,\xi] &\mapsto [y,\xi] \mapsto [[y,1_{G^{\BC}}],\xi] \mapsto [sg,\xi] = [s,\Ad(g)\xi] \mapsto (x,\Ad(g)\xi) \label{grothendieckisomorphisms}
\end{align}
Here, recall that $s : x \mapsto (x,1_{G^{\BC}})$ is our chosen section of $P$. Note also that if we had chosen $[yz,\xi]$ instead, we would map to $(x,\alpha(\chi(z))\Ad(g)\xi)$. In particular, the line is preserved. This means that one can identify a line $L_{x} \seq \mathfrak{g}^{\BC}$ corresponding to each point of $S^{2}.$

With this technical issue resolved, let us see what $q$ can tell us. Note that $q$ is smooth because $L$ is a smooth sub-bundle of the product. Next, let $e$ be the same element of $\mathfrak{g}^{\theta}$ from Section \ref{liealgebrasection}. The image of $q$ is a subset of the Adjoint orbit $\Ad(G^{\BC})[\Span_{\BC}\{e\}] \seq \mathbb{P}(\mathfrak{g}^{\BC}).$ From Lemma \ref{stablemma}, one sees that at the Lie algebra level the stabilizer is $\mathfrak{q}_{\theta} = \mathfrak{g}_{0} \oplus \mathfrak{g}_{1} \oplus \mathfrak{g}_{2}$. Recall that this is parabolic with associated subgroup $P_{\theta}$. This means that $q$ is valued in the compact flag manifold $G^{\BC}/P_{\theta}.$ Recall also that the compact real form acts transitively on the flag manifold, so $q$ can be viewed as $G/M_{\theta}$-valued where $M_{\theta} = G \cap P_{\theta}.$ Let $\mathfrak{m}_{\theta}$ denote the Lie algebra of $M_{\theta}$.

We conclude that a local section of the bundle $q^{*}(G \to G/M_{\theta})$ yields a map $k_{i} : U_{i} \seq S^{2} \to G$ such that $L_{x} \seq \mathfrak{g}^{\BC}$ can be written explicitly as $\Ad_{k_{i}(x)}(\Span_{\BC}\{e\}).$ This says moreover that $L$ can be written locally via a compact-valued gauge. Note that any two $k_{i}$ and $k_{j}$ differ on the intersection of their domain by an $M_{\theta}$-valued map because $M_{\theta}$ is the structure group of $q^{*}(G \to G/M_{\theta})$. We call these $k_{i}$ ``compact gauges."

With these local gauges in hand, we may define our variation. First, set $H = [e,e^{\dagger}]$ as before. We wish to use $H$ to construct our variation. To do this we use our compact gauges. Set 
\begin{align}
e_{i}(x) = \Ad_{k_{i}(x)}e. 
\end{align}
We then set 
\begin{align}
H_{i}(x) = [e_{i}(x),e_{i}(x)^{\dagger}]. 
\end{align}
Note we can exchange $\Ad$ with $\dagger$ because $k_{i}$ is $G$-valued and $\Ad$ is linear. Now let us glue $H_{i}$ to a global section. We have 
\begin{align}
e_{j}(x) = \Ad_{k_{j}(x)}e = \Ad_{k_{i}(x)m_{ij}(x)}e. 
\end{align}
Then we see that $\Ad_{m_{ij}(x)}e = ce$ with $|c| = 1$ because $m_{ij}$ is part of the stabilizer of $[\Span_{\BC}\{e\}]$ and $\Ad$ is an isometry on the compact group. This says that
\begin{align}
H_{j}(x) = [ce_{i}(x),\bar{c}e_{i}(x)^{\dagger}] = H_{i}(x).
\end{align}
Therefore, $H$ defines a global section. Note that the same argument shows that for $m \in M_{\theta}$, we have $\Ad(m)H = H$. We see $H^{\dagger} = H$, so $\sqrt{-1}H \in \mathfrak{g}.$ We therefore take the variation 
\begin{align}
\sigma : S^{2} \to \mathfrak{g} : x \mapsto \sqrt{-1}H(x). 
\end{align}
By construction, $H(x) = \Ad_{k}H$ for some $k \in G$, so $|H(x)|$ is constant. We computed before that $|H| ^{2}= 2$, so $|\sigma|^{2} = 2$ as well. 

\section{Writing $d + \ad(\frac{1}{2}\alpha)$ in the Compact Gauge}\label{b-2vanishessection}
Here we will work entirely in the product bundle $S^{2} \times \mathfrak{g}^{\BC} \to S^{2}$. Recall our connection on this bundle is $d + \ad(C).$ Choose a compact gauge of the form above called $k : U \seq S^{2} \to G.$ In this gauge, $C$ transforms to $b = \Ad(k^{-1})C + k^{-1}dk.$ The bundle-valued component of $\Phi$ transforms as $\Ad(k^{-1})A = A_{k}$. One verifies that in this gauge:
\begin{align}
F_{b,z\bar{z}} &= [A_{k},A^{\dagger}_{k}] \label{compactcurv}\\
b_{z} &= -b_{\bar{z}}^{\dagger}
\end{align}

Let $s$ be a section of the product bundle defined by $s = \Ad(k)e.$ Then once we change into the compact gauge, we get $s_{k} = e.$ We differentiate $s_{k}$ in our compact gauge:
\begin{align}
D^{0,1}s_{k} = \bar{\p}s_{k} + [b_{\bar{z}},s_{k}]d\bar{z} = \bar{\p}e + [b_{\bar{z}},e]d\bar{z} = [b_{\bar{z}},e]d\bar{z}
\end{align}
On the other hand, $s$ is a section of the holomorphic sub-bundle $L$ by construction, so $D^{0,1}s$ is a $(0,1)$-form with values in $L.$ Therefore, $D^{0,1}s_{k} = \gamma e$ for some $\gamma \in \Omega^{0,1}(U;\BC).$ Thus, $[b_{\bar{z}},e] = \gamma_{\bar{z}} e.$ This implies $b_{\bar{z}} \in \mathfrak{g}_{0} \oplus \mathfrak{g}_{1} \oplus \mathfrak{g}_{2}.$ By complex conjugation we can then write $b_{z} = b_{z,0} + b_{-1} + b_{-2}.$ Our goal is to show $b_{-2} = 0.$ 

Before this, also note that $[\Phi,L] = 0$ tells us that $[C_{z},\Ad(k)e] = 0$, so $\Ad(k)[A_{k},e] = 0$, i.e., $[A_{k},e] = 0.$ This means that $A_{k} \in \mathfrak{c} \oplus \mathfrak{g}_{1} \oplus \mathfrak{g}_{2}$ because $\mathfrak{g}_{0} = \Span_{\BC}\{H\} \oplus \mathfrak{c}$ where $\mathfrak{c}$ is the centralizer of $V$. Therefore, we can write $A_{k} = A_{0} + A_{1} + A_{2}$ where $A_{0} \in \mathfrak{c}$.

\begin{proposition}\label{b-2vanishesincptgauge}
In any compact gauge, $b_{-2} = 0.$
\end{proposition}
\begin{proof}
Let $E_{j}$ be the smooth sub-bundle of $S^{2} \times \mathfrak{g}^{\BC} \to S^{2}$ defined by conjugating $\mathfrak{g}_{j}$ by $k.$ Thus, in our compact gauge $E_{j}$ is just $\mathfrak{g}_{j}.$

To prove the proposition, we first show that $\beta = b_{-2}dz$ defines a global section of $E_{-2} \otimes K.$ We then identify $E_{-2}$ smoothly with $L^{*}$ and another bundle $Q.$ Finally, we show $\beta$ is holomorphic as a section of $Q \otimes K$ and hence zero for degree reasons. Note that $E_{-2}$ is not necessarily a holomorphic sub-bundle of $S^{2} \times \mathfrak{g}^{\BC}$, and $E_{-2}$ is also not necessarily the root line bundle corresponding to $-\theta.$ It is just the conjugation of our fixed root space $\mathfrak{g}_{-2}$ in this local trivialization. Therefore, we do need to identify it with a holomorphic bundle.

To show $\beta$ defines a global form, we check how it changes under a change of trivialization. In particular, if $k' = km$ then the connection coefficient transforms as
\begin{align}
b_{z}' = \Ad(m^{-1})b_{z} + m^{-1}\p_{z}m. 
\end{align}
Now, for any tangent vector field $X_{x}$ at $x \in U \cap U' \seq S^{2}$ we have:
\begin{align}
m^{-1}dm(X) \in \mathfrak{m_{\theta}} \seq \mathfrak{g}_{0} 
\end{align}
Therefore, the $-2$ component of $b$ transforms correctly (we need $\Ad(k)b_{-2} = \Ad(k')b_{-2}'$, i.e., $b_{-2} = \Ad(m)b_{-2}'$ on $U \cap U'$), so we get a global section $\beta \in \Gamma(E_{-2} \otimes K).$

Next, set $L^{\perp_{\kappa}} = \{(x,X) \in S^{2} \times \mathfrak{g}^{\BC} : \kappa(X,L_{x}) = 0\}$, $Q = (S^{2} \times \mathfrak{g}^{\BC})/L^{\perp_{\kappa}}$, and $\pi : S^{2} \times \mathfrak{g}^{\BC} \to Q.$ By looking at the expression for $D\kappa$, one checks that $\kappa$ is parallel. From this, we take the $(0,1)$-part of the identity to conclude that $\bar{\p}(\kappa(X,Y)) = \kappa(D^{0,1}X, Y) + \kappa(X,D^{0,1}Y).$ Let $X \in \Gamma(L^{\perp_{\kappa}})$ and $Y \in \Gamma(L).$ These are $\kappa$-orthogonal, so one obtains
\begin{align}
\kappa(D^{0,1}X,Y) + \kappa(X,D^{0,1}Y) = 0.
\end{align}
Then recall that $D^{0,1}Y \in \Gamma(L) \otimes \Omega^{0,1}$ because $L$ is a holomorphic sub-bundle. This implies $\kappa(D^{0,1}X,Y) = 0$. In other words, $L^{\perp_{\kappa}}$ is a holomorphic sub-bundle of $S^{2} \times \mathfrak{g}^{\BC}.$ This induces a holomorphic structure on $Q$ as a quotient. The associated Dolbeault operator is $D^{0,1}_{Q}(\pi(X)) = \pi(D^{0,1}X).$ 

Next, define the bundle map $\psi : S^{2} \times \mathfrak{g}^{\BC} \to L^{*}$ by $\psi(X)(\ell) = \kappa(X,\ell).$ One has $\ker \psi = L^{\perp_{\kappa}}$, so $L^{*}$ is smoothly isomorphic to $Q$. One also has a smooth isomorphism of $E_{-2}$ and $Q$ given by the projection map $\pi$. This tells us in particular that $\beta = 0$ if and only if $\pi(\beta) = 0.$

Now, set $\beta_{Q} = \pi(\beta).$ We want to show $D^{0,1}_{Q \otimes K}\beta_{Q} = 0.$ To do this, first recall that
\begin{align}
F_{b,z\bar{z}} = \p_{z}b_{\bar{z}} - \p_{\bar{z}}b_{z} + [b_{z},b_{\bar{z}}].
\end{align} 
On the other hand, we know this must be $[A_{k},A_{k}^{\dagger}]$, which can only have $-2$-component $[A_{0},A_{2}^{\dagger}] = 0$. This must be zero because $A_{0} \in \mathfrak{c}$. Note that differentiating preserves the grades in this frame. Thus, $\p_{\bar{z}}b_{-2} + [b_{\bar{z},0}, b_{-2}] = 0.$ It turns out that this will be enough to show that $\beta_{Q}$ is holomorphic. To this end, in our frame locally,
\begin{align}
D^{0,1}_{Q \otimes K}(\beta_{Q}) &= \pi(D^{0,1}b_{-2})dz
\end{align}
We then compute $\pi(D^{0,1}b_{-2})$:
\begin{align}
\pi(D^{0,1}b_{-2}) &= \pi(\bar{\p}b_{-2} + [b_{\bar{z}}d\bar{z}, b_{-2}]) \\ \non \\
&= \pi(\p_{\bar{z}}b_{-2} + [b_{\bar{z},0},b_{-2}] + [b_{\bar{z},1},b_{-2}] + [b_{\bar{z},2}, b_{-2}])d\bar{z} \\ \non \\
&= \pi(\p_{\bar{z}}b_{-2} + [b_{\bar{z},0},b_{-2}]) d\bar{z} = 0.
\end{align}
This implies that $\beta_{Q}$ is holomorphic. Now notice that 
\begin{align}
c_{1}(Q) = c_{1}(L^{*}) = -c_{1}(L) \leq -2
\end{align}
Then we see $c_{1}(Q \otimes K) \leq -4 < 0$, so $\beta_{Q} \in H^{0}(S^{2},Q \otimes K) = 0$. Therefore, $\beta_{Q} = 0$, so $\beta = 0$ as well. This tells us that $b_{-2}$ must always vanish in a compact gauge.
\end{proof}

\section{A Topological Formula for the $L^{2}$ Norms of Components of $A$ and $b$}
As in \cite{Kra26}, the key to our choice of variation is it enables a clean evaluation of the stability inequality. In the present case, it becomes important to estimate the $L^{2}$ norm of certain components of $A$ and $b.$ The main tool we have is a clean formula for the curvature of our connection.

Thus, to begin, recall from Chern-Weil Theory that
\begin{align}
c_{1}(L) = \dfrac{1}{\pi}\int_{S^{2}}F_{L,z\bar{z}}dxdy
\end{align}
Therefore, by (\ref{compactcurv}) we might hope that $F_{L,z\bar{z}}$ is expressible in terms of the quantities $A$ and $b.$

To see that this is the case, we need to choose a connection on $L.$ We choose the Chern connection on $L$, which is the orthogonal projection of the connection on $S^{2} \times \mathfrak{g}^{\BC}$ onto $L.$ In particular, in our compact gauge $k$, this is projection onto $\mathfrak{g}_{2}.$ Therefore, $D_{L}e = (de)_{2} + [b,e]_{2} = [b_{0},e].$ This means the connection on $L$ is just $d + \ad(b_{0})$ in our compact gauge. However, notice that $\ad(b_{0})$ acts on $e$ by $\lambda : \mathfrak{g}_{0} \to \BC$ which is defined by $\ad(X)e = \lambda(X)e$ because $\mathfrak{g}_{0} = \mathfrak{c} \oplus \operatorname{span}_{\BC}\{H\}$. This tells us that the connection form in this gauge is written $a = \lambda(b_{0}).$ The $[a,a]$ term of the curvature vanishes because $a$ is $\BC$-valued. We conclude: 
\begin{align}
F_{L,z\bar{z}} = \p_{z}a_{\bar{z}} - \p_{\bar{z}}a_{z} = \lambda(\p_{z}b_{\bar{z},0} - \p_{\bar{z}}b_{z,0})
\end{align}
Note that the second equality follows from linearity of $\lambda$.

We now attempt to recover this formula within the curvature of $C$. To this end, recall that we have two different formulas for computing $F_{b,z\bar{z}}.$ The first is the definition:
\begin{align}
F_{b,z\bar{z}} = \p_{z}b_{\bar{z}} - \p_{\bar{z}}b_{z} + [b_{z},b_{\bar{z}}].
\end{align}
Looking at this expression carefully, we see that the first two terms are precisely the curvature of $L$ once we take the $0$ component and apply $\lambda.$ Therefore, we must try to understand $\lambda([b_{z},b_{\bar{z}}]_{0}).$ We will compute this directly. First we should get a nicer expression for $\lambda:$
\begin{align}
\lambda(X) = \lambda(X_{\mathfrak{c}} + cH) = c\lambda(H) = 2c \implies \lambda(X) = \kappa(X,H).
\end{align}
We would like to apply this formula to $[b_{z},b_{\bar{z}}]_{0}$. Recall again that $b_{z} = -b_{\bar{z}}^{\dagger}$ because it is the complex linear extension of a real form. By Proposition \ref{b-2vanishesincptgauge} of the last section, we see that $b_{z} = b_{z,0} + b_{-1}.$ This tells us that $b_{\bar{z}} = -b_{z,0}^{\dagger} - b_{-1}^{\dagger}.$ Using this, we conclude that
\begin{align}
[b_{z},b_{\bar{z}}]_{0} = [b_{z,0},-b_{z,0}^{\dagger}] + [b_{-1},-b_{-1}^{\dagger}]
\end{align}
Now notice that if $X \in \mathfrak{g}_{k}$, we can compute
\begin{align}
\lambda([X,X^{\dagger}]) &= \kappa([X,X^{\dagger}],H) \\ \non \\
&= -\kappa(X^{\dagger},[X,H]) \\ \non \\
&= k\kappa(X,X^{\dagger}) \\ \non \\
&= k|X|^{2}.
\end{align}
From here, we see that $\lambda([b_{z},b_{\bar{z}}]_{0}) = |b_{-1}|^{2}.$ Next, we reach for the formula (\ref{compactcurv}). Recall that 
\begin{align}
A_{k} = A_{0} + A_{1} + A_{2} \in \mathfrak{c} \oplus \mathfrak{g}_{1} \oplus \mathfrak{g}_{2}.
\end{align}
 This means that
\begin{align} 
[A_{k},A_{k}^{\dagger}]_{0} = [A_{0},A_{0}^{\dagger}] + [A_{1},A_{1}^{\dagger}] + [A_{2},A_{2}^{\dagger}].
\end{align}
Applying $\lambda$, we obtain $\lambda([A_{k},A_{k}^{\dagger}]_{0}) = |A_{1}|^{2} + 2|A_{2}|^{2}.$

Looking at both of our formulas we see that $F_{L,z\bar{z}} = |A_{1}|^{2} + 2|A_{2}|^{2} - |b_{-1}|^{2}.$ Plugging into the Chern-Weil identity, we obtain:
\begin{align}
\pi c_{1}(L) = \int_{S^{2}}|A_{1}|^{2} + 2|A_{2}|^{2} - |b_{-1}|^{2}dxdy \label{curvature identity}
\end{align}
This will be incredibly useful in the following section. Note that it does take some care to establish that $A_{j}dz$ and $b_{-1}dz$ patch to global forms, but the argument is similar to the case of $b_{-2}dz.$

\section{$\sigma$ Fails Cone Stability}
Now we have all the tools we need to show that $\sigma$ fails the cone stability inequality. First recall that the cone stability inequality says that
\begin{align}
\dfrac{1}{4}\int_{S^{2}}|X|^{2}\operatorname{vol}_{S^{2}} + \int_{S^{2}}|DX|^{2} - g(R(X,\alpha_{i})\alpha_{i},X)\operatorname{vol}_{S^{2}} \geq 0
\end{align}
for every $X \in \Gamma(S^{2}, S^{2} \times \mathfrak{g})$, where $D = d + \ad C$. Here $\alpha$ is again the pullback of the left Maurer-Cartan form by $\phi.$
\begin{proposition}\label{varfailsconestab}
The variation $\sigma$ fails the cone stability inequality.
\end{proposition}
\begin{proof}
Assuming our map $\phi$ is cone stable, let us evaluate the cone stability inequality on $\sigma$ piece by piece. The first term is just:
\begin{align}
\dfrac{1}{4}\int_{S^{2}}|\sigma|^{2}\operatorname{vol}_{S^{2}} = \dfrac{1}{4}(4\pi)(2) = 2\pi.
\end{align}
For the second integral, we recognize this is the usual index form of $\phi$, $I(\sigma,\sigma)$. We write this down in terms of the complex components of our forms (see for instance Micallef and Moore \cite{MM88}, Burstall and Rawnsley \cite{BR90}):
\begin{align}
I(\sigma, \sigma) = 4\int_{S^{2}}|D_{z}\sigma|^{2} - |[A, \sigma]|^{2}dxdy
\end{align}
Notice that if we change into the compact gauge $k$, $D\sigma = \Ad(k)D_{k}\sigma_{k}$. This means the norm will not change in the compact gauge.
\begin{align}
D_{k,z}\sigma_{k} = \p_{z}\sigma_{k} + [b_{z},\sigma_{k}] = [b_{-1},\sigma_{k}] = \sqrt{-1}b_{-1}.
\end{align}
Thus,
\begin{align}
|D_{z}\sigma|^{2} = |b_{-1}|^{2}.
\end{align}
Lastly we compute the curvature term:
\begin{align}
|[A,\sigma]|^{2} &= |\Ad(k)[A_{k},\sigma_{k}]|^{2} \\ \non \\
&= |-\sqrt{-1}A_{1} -2\sqrt{-1}A_{2}|^{2} \\ \non \\
&= |A_{1}|^{2} + 4|A_{2}|^{2}.
\end{align}
If our map $\phi$ is cone stable, applying (\ref{curvature identity}) from the last section we obtain:
\begin{align}
\dfrac{1}{4}\int_{S^{2}}|\sigma|^{2}\operatorname{vol}_{S^{2}} + I(\sigma,\sigma) &= 2\pi + 4\int_{S^{2}}|b_{-1}|^{2} - |A_{1}|^{2} - 4|A_{2}|^{2}dxdy \\ \non \\
&= 2\pi - 4\pi c_{1}(L) - 8\int_{S^{2}}|A_{2}|^{2}dxdy \geq 0
\end{align}
We conclude that
\begin{align}
2\pi(1 - 2c_{1}(L)) - 8\int_{S^{2}}|A_{2}|^{2}dxdy \geq 0. 
\end{align}
Recalling (\ref{c_1(L)>=2}), we reach a contradiction.
\end{proof}
\section{Proof of the Main Theorem}
We briefly recall the main theorem. Let $u : M \to G$ be a stable stationary harmonic map into a compact Lie group with bi-invariant metric. We claim that $u$ is smooth away from a set of Hausdorff codimension at least four.
\begin{proof}
By Hsu's dimension reduction theorem, we just need to show there are no nonconstant cone stable $S^{2} \to G.$ Suppose there is a nonconstant cone stable $\phi : S^{2} \to G.$ First, translate $\phi$ to have image in the identity component of $G.$ This translated map satisfies the same properties as the original map. Therefore, we may assume that $G$ is connected. The result then follows immediately from Proposition \ref{varfailsconestab} and the lifting argument in section nine of \cite{Kra26}.
\end{proof}

\bigskip

\bibliography{generalcodim4}{}
\bibliographystyle{plain}

\end{document}